\documentclass[12pt]{amsart}

\usepackage{
amsfonts,
amsmath,
latexsym,
amssymb,
enumerate,
mathrsfs,
color,
}

\usepackage[matrix,arrow]{xy}

\newcommand{\labbel}{\label}

\newcommand{\arxiv}{}

\newcommand{\sharr}{\text{\tiny{$\rightarrow$}}}

\newtheorem{theorem}{Theorem}[section]
\newtheorem{lemma}[theorem]{Lemma}

\newtheorem*{claim*}{Claim}

\newtheorem*{theorem*}{Theorem}
\newtheorem*{proposition*}{Proposition}
\newtheorem*{corollary*}{Corollary}
\newtheorem*{lemma*}{Lemma}
\newtheorem*{scholion*}{Scholion}

\theoremstyle{definition}
\newtheorem{definition}[theorem]{Definition}

\theoremstyle{remark}
\newtheorem{remark}[theorem]{Remark}

\newtheorem*{remark*}{Remark}
\newtheorem*{remarks*}{Remarks}

\newtheorem*{observation*}{Observation}

\numberwithin{equation}{section}

\begin{document}

\title
{Universal extensions of specialization semilattices}

\author{Paolo Lipparini} 

\email{lipparin@axp.mat.uniroma2.it}

\urladdr{http://www.mat.uniroma2.it/~lipparin}

\address{Dipartimento di Matematica\\Viale della  Ricerca
 Scientifica Non Chiusa \\Universit\`a di Roma ``Tor Vergata'' 
\\I-00133 ROME ITALY}



\thanks{Work performed under the auspices of G.N.S.A.G.A. 
Work partially supported by PRIN 2012 ``Logica, Modelli e Insiemi''.
The author acknowledges the MIUR Department Project awarded to the
Department of Mathematics, University of Rome Tor Vergata, CUP
E83C18000100006.}

\begin{abstract}
A specialization semilattice is a join 
semilattice together with a coarser  preorder 
$ \sqsubseteq $  satisfying an appropriate compatibility condition.
If $X$ is a topological space,
then $(\mathcal P(X), \cup, \sqsubseteq )$
is a specialization semilattice, where
$ x \sqsubseteq y$ if $x \subseteq Ky$,
   for $x,y \subseteq  X$, and $K$ is closure.
Specialization semilattices and posets
appear as auxiliary structures in many 
disparate scientific fields, even unrelated to topology.

\arxiv{For short, the notion is useful since it allows us to 
consider a relation
of ``being generated by'' 
with no need to require the existence of an actual  ``closure'' 
or  `` hull'', which is problematic
in certain contexts.} 

In a former work we showed
that every specialization semilattice can be embedded
into the specialization semilattice associated
to a topological space as above. 
Here we describe the universal embedding
of a specialization semilattice into an
additive 
 closure semilattice.
We notice that a categorical argument guarantees
the existence of  universal embeddings in  many parallel situations.
\end{abstract}

\maketitle

\arxiv{
\section{Specialization without actual closure} \labbel{introa} 
 } 

The idea of \emph{closure} 
is pervasive in mathematics.
First, the notion is used in the sense of 
\emph{hull}, \emph{generated by},
for example when we consider the subgroup  
generated by a given subset of some group.
In a slightly different but related sense,
closure is a fundamental notion in topology.
In both cases, ``closed'' sets are preserved under
arbitrary intersections; in the topological case
the union of two closed sets is still closed; in most
``algebraic'' examples,  the union of an upward directed family
of closed subsets is still closed.

The general notion of a \emph{closure space}
which can be abstracted from the above examples
has been dealt with or foreshadowed by 
such mathematicians as 
  Schr\"oder, Dedekind, Cantor, Riesz, Hausdorff, Moore,
\v Cech,  Kuratowski, Sierpi\'nski, Tarski, Birkhoff and
Ore, as listed in Ern\'e \cite{E}, with applications, 
among others, to ordered sets, lattice theory,
logic, algebra,  topology, computer science
and connections with category theory.
\arxiv{
See  the mentioned work \cite{E}   for more details and references.

Considering ``full'' closure sometimes generates
objects that are  ``too large''.
For example, if we are working with sets of groups
and we consider, as closure,
the operation of taking arbitrary products of
members of 
the set under consideration, then the resulting operation
takes a set to a proper class, an object which 
causes foundational issues.

As a smaller and more concrete example, suppose that 
we are given a finite presentation 
by generators and relations of some group $\mathbf G$,
and $\mathbf H$ is the subgroup generated by a finite set $F \subseteq G$. 
It might turn out that $G$ and $H$ are actually infinite, hence we cannot 
store the list of
\emph{all} the elements of $H$, say, in the hard disk of a computer.
However, we can store the information that some finite set
$E$ is contained in $H$.
Thus the set $H= \langle F \rangle $
is too large to be actually listed, while the information 
that everything in $E$ can be generated by $F$  turns out to be more tractable.
In other words, we do not need to consider 
$\langle F \rangle$ as ``realized''
if we are only interested in the binary relation
$ E \sqsubseteq F$ given by   $E \subseteq \langle F \rangle$,
for $E, F$ finite subsets of $G$.

In many cases we are in a similar situation:
it is not necessary to describe the actual closure,
we just need to know whether some object is contained or 
not  in  the closure.
Turning to an example above, we generally do not need to consider
 the class of all groups which can be 
expressed as products of a given set of groups
(as we mentioned, a problematic object, anyway).
We usually simply need to know that some specific group
can be expressed in such a way.
} 

Even in topology, one frequently needs to consider
only the \emph{adherence} relation $p \in K y$,
meaning that the element $p$ belongs to the   
topological closure of the subset $y$\arxiv{, with no need to deal with the full closure $Ky$}. 
Arguing in terms of adherence provides a conceivably
more intuitive approach to continuity:
as well-known, a function $f$ between topological spaces is continuous
if and only if $f$ preserves the adherence relation, namely,
if and only if  $p \in K y$ implies 
$f(p) \in K f(y)$.

Similarly, we can consider the \emph{specialization} relation
$ x \sqsubseteq y$ defined by 
$x \subseteq Ky$, for $x,y$ subsets of some topological space $X$.
It is a natural 
generalization of the \emph{specialization preorder}
defined on points
of a topological space  \cite[Ex. 3.17e]{Ha}, \cite{GHKLMS}.
As in the case of adherence,
 a function $f$  from $X$ to some other space $Y$ is continuous
if and only if the direct image function $f^\sharr$  is a homomorphism
from the structure $(\mathcal P(X), \cup, \sqsubseteq )$
to  $(\mathcal P(Y), \cup, \sqsubseteq )$.

The above ``algebraization'' of topology is thus significantly
different from the classical  approach presented in \cite{MT},
where the operation $K$ of closure is taken into account.
The notion of homomorphism in \cite{MT} does not correspond to the notion of
continuity. In fact,  a function $f$ between two spaces 
is continuous if and only if  $f^\sharr(Kx) \subseteq Kf^\sharr(x)$,
for all subsets $x$. On the other hand, a homomorphism $\varphi$  of closure algebras 
\cite{MT} is assumed to satisfy the stronger condition
$\varphi (Kx)=K\varphi(x)$.

In \cite{mtt} we characterized \emph{specialization semilattices},
those structures which can be embedded into   
$(\mathcal P(X), \cup, \sqsubseteq )$ for some topological space $X$,  
and \emph{specialization posets}, which can be embedded into   
$(\mathcal P(X), \subseteq , \sqsubseteq )$.
See \eqref{s1}\,-\,\eqref{s3} below. 
While our main interest was algebraic and model-theoretical,
we realized that such structures appear in many distinct and unrelated settings.

A typical example of a specialization to which no closure
can be associated is \emph{inclusion modulo finite}.
If $X$ is an infinite set and we let $ x \sqsubseteq y$
if $x \setminus y$ is finite, for $x,y \subseteq X$,
then $(\mathcal P(X), \cup, \sqsubseteq )$
is a specialization semilattice. Inclusion modulo finite
plays important roles, among other,
 in  set theory, topology and  model theory \cite{Bl,MN}.
From a slightly different  perspective,
working modulo finite corresponds to taking
the quotient modulo the ideal of finite sets on the
standard Boolean algebra on $\mathcal P(X)$.
From the present point of view, a similar construction
can be used to generate specialization semilattices: 
if 
$\varphi: \mathbf S \to \mathbf T $  is a semilattice homomorphism
 and we set $ a \sqsubseteq  b$ in $S$ 
when $\varphi(a) \leq \varphi (b)$
in $\mathbf T$, then 
$\mathbf S$ 
is endowed with the structure of a specialization 
semilattice. As  we shall show elsewhere,
 every specialization semilattice  
can indeed be constructed this way. 
Specialization semilattices are
substructures of topological spaces 
in the language with union and  $ \sqsubseteq $, 
but, in a sense, they are also semilattices 
together with  a quotient (or a congruence). 
 
Under different terminology, specialization appears in \cite{GT}  
in the context of complete lattices,
with deep and important applications 
to algebraic logic.
See  Conditions (1) - (2)
in \cite[Subsection 3.1]{GT}. 
Specialization semilattices arise also naturally in the theory of
\emph{tolerance spaces} \cite{PN}, with applications
 to image
analysis  and  information systems
\cite{PW}.   

\emph{Causal spaces}
have been introduced  by 
Kronheimer and  Penrose  \cite{KP}
 in connection with 
abstract foundations of 
general relativity.
Causal spaces
can be axiomatized as
two partial orders, one finer than the other,
and satisfying a further coherence condition. In particular,
they are specialization posets.
As another example, if $\mu $ is a measure on some set $S$ of subsets
of $X$, then
$a \sqsubseteq_ \mu b$
defined by  $\mu (a) \leq \mu (b)$, for $a,b \in S$,
is a preorder which forms a specialization
poset together with inclusion.
 If $\mu $ is 
$2$-valued, then we get a specialization semilattice. 
Such structures have been widely studied
in connection with foundations of probability.
See \cite{L} and references there.  

A \emph{closure poset}  (\emph{semilattice}) is a 
partially ordered set (join semilattice) together 
with an isotone, extensive and
idempotent  operator $K$.
See Remark \ref{psc}.
If $K$ satisfies
$K(a \vee b) = Ka \vee Kb$
and $K0=0$ 
in a closure semilattice with minimum $0$, then  
 $K$ satisfies the Kuratowski axioms
for topological closure.
Closure posets and semilattices
have many applications; see \cite{E,R} 
for references.
As in the case of topological spaces,
setting $ a \sqsubseteq b$ if $a \leq Kb$
induces the structure of a specialization poset (semilattice)
and a large part of the theory of closure posets
applies to this more general setting. 
See  \cite{mtt} for more details and further examples. 

Henceforth we were convinced that the notion of a specialization semilattice
deserves an accurate study, both for its possible foundational relevance
in connection with topology, and since the notion appears
in many disparate fields.
\arxiv{The main reason for the latter fact is possibly 
the need or the opportunity, as singled out at the beginning
of this introduction, 
 of asserting that some object 
belongs to the hull generated by another object
without having to  deal with full ``closure''. } 

The main result in \cite{mtt} asserts that every 
specialization semilattice or poset can be embedded into a
``topological'' one.
The extensions constructed in \cite[Section 4]{mtt}
are not minimal and  possibly neither canonical nor 
 functorial. In  search for a better-behaved
extension, here we 
explicitly describe the universal embedding
of a specialization semilattice into an
additive closure semilattice.
This is done in Section \ref{univ1}. 
In Section \ref{univ} we  then show that the existence of such an embedding,
as well as the existence of a multitude of other embeddings,
follow from an abstract argument.

\section{Preliminaries} \labbel{prel}

A \emph{specialization semilattice} \cite[Definition 3.1]{mtt} 
is a join semilattice endowed with
a further preorder $ \sqsubseteq $ 
which is coarser than the order $\leq$ induced by 
$\vee$ and satisfies the further compatibility relation
\eqref{s3} below.  
In detail, a specialization semilattice $\mathbf S$ 
is a triple 
$(S, \vee, \sqsubseteq ) $
such that 
$(S, \vee)$
is a semilattice
and moreover
\begin{align}
\labbel{s1}    \tag{S1}
 & a \leq b  \Rightarrow  a \sqsubseteq b, 
\\
\labbel{s2}    \tag{S2}
& a \sqsubseteq b \ \&\   b \sqsubseteq c \Rightarrow 
 a \sqsubseteq c, 
\\
\labbel{s3}    \tag{S3}
 &a \sqsubseteq b  \ \&\  a_1 \sqsubseteq b
 \Rightarrow 
 a \vee a_1 \sqsubseteq b, 
\end{align}
 for all elements $a,b,   c, a_1  \in S$.
Notice that from \eqref{s1} one  gets 
\begin{align} 
\labbel{s4}    \tag{S4} 
 & a  \sqsubseteq a,
 \end{align}  
for every $a $ in $ S$. 

It can be shown \cite[Remark 3.4(a)]{mtt} that every specialization semilattice 
satisfies
\begin{equation}  
\labbel{s7}    \tag{S7}
 a \sqsubseteq b  \ \&\  a_1 \sqsubseteq b_1
  \Rightarrow 
 a \vee a_1 \sqsubseteq b \vee b_1 .
 \end{equation} 

A \emph{specialization poset}  
is a partially ordered set with a further preorder satisfying
\eqref{s1} - \eqref{s2}. Specialization posets occur naturally
in many situations, but the theory of specialization semilattices
is much cleaner and here we shall be mainly interested in the latter.  

A \emph{homomorphism} of specialization semilattices
is a semilattice homomorphism $\eta$
such that $ a \sqsubseteq b$ implies $\eta(a) \sqsubseteq  \eta(b)$.
An \emph{embedding} is an injective homomorphism
satisfying the additional condition that
$\eta(a) \sqsubseteq  \eta(b)$ implies $ a \sqsubseteq b$.

If $\mathbf S$ is a specialization semilattice,
$a \in S$ and the set
$S_a=\{ b \in S \mid  b \sqsubseteq a \}$
has a $\leq$-maximum, such a maximum shall be denoted
by $Ka$ and shall be called the \emph{closure} of $a$.   
Notice that we require $Ka$ to be the maximum
of $S_a$, not just a supremum. Namely, we 
require $Ka \sqsubseteq a$. 

In general, $Ka$ need not exist
in an arbitrary specialization semilattice:
consider the example of inclusion modulo finite
mentioned in  the introduction.
If $Ka$ exists for every  
$a \in S$, then $\mathbf S$ shall be called a
\emph{principal} specialization semilattice.

\begin{remark} \labbel{psc}    
(a) Principal specialization semilattices
are in a one-to one correspondence with
\emph{closure semilattices}, that is, semilattices with a further
operation $K$ such that 
$a \leq Ka$, $KKa=Ka$,
  and $K(a \vee b) \geq Ka \vee Kb$.

If $\mathbf  C$ is a closure semilattice, then
setting $ a \sqsubseteq b$
if $a \leq Kb$ makes $\mathbf  C$ 
a specialization semilattice,
and $K$  turns out to 
be closure also
   in the  sense of specialization semilattices.
See \cite[Section 3.1]{E}, in particular,
\cite[Proposition 3.9]{E} for details.     
 
(b) The clause $K(a \vee b) \geq Ka \vee Kb$
is  equivalent to the condition that
$c \geq a$ implies $Kc \geq Ka$.  
As a consequence, we get
$K(a \vee b) \leq  K(a \vee Kb) $
in closure semilattices.
Moreover,  
$K(a \vee b) \geq Ka \geq  a  $,
$K(a \vee b) \geq   Kb $, so
$K(a \vee b) \geq  a \vee Kb $,
hence
$K(a \vee b)=KK(a \vee b) \geq  K(a \vee Kb )$.
In conclusion, as well-known,
$K(a \vee b) = K(a \vee Kb) $
in every closure semilattice.

By the same argument, we could even
prove $K(a \vee b) = K(Ka \vee Kb) $,
but we shall not need this in what follows.
 
(c) If $a$ and $b$ are elements of  some specialization semilattice 
and both $Ka$ and  $Kb$ exist, then
$Ka \leq Kb$ if and only if $ a \sqsubseteq b$.
Indeed, from $a \leq a$ and \eqref{s1} we get
$ a \sqsubseteq a$, thus $a \leq Ka$, by the definition
of $Ka$.
Hence if   $Ka \leq Kb$, then $a \leq Kb$,
which means $a \sqsubseteq b$, by the definition 
of $Kb$.
Conversely, if $ a \sqsubseteq b$,
then from $ Ka \sqsubseteq a$ and 
\eqref{s2} we get   $ Ka \sqsubseteq b$,
which means $Ka \leq Kb$.

In particular, in a principal specialization semilattice,
$Ka=Kb$ if and only if both $a \sqsubseteq b$ and  
$b \sqsubseteq a$. 
\end{remark}

If $\mathbf S$ and $\mathbf T$ are principal specialization
semilattices, a \emph{$K$-{\hspace{0 pt}}homomorphism} from
$\mathbf S$ to $\mathbf T$
is a homomorphism  $\eta$
which
preserves $K$, that is  $\eta (Ka)= K\eta(a)$.
Thus $K$-homomorphisms correspond to the natural
 notion of 
homomorphism for closure semilattices.

If $\eta$ is a semilattice homomorphism
between two principal specialization semilattices and $\eta$ 
satisfies $\eta (Ka)= K\eta(a)$, then $\eta$ 
is also a $ \sqsubseteq $-{\hspace{0 pt}}homomorphism. 
Indeed,
$ a \sqsubseteq b$ is equivalent to $a \leq Kb$,
hence $ \eta (a) \leq \eta (Kb)= K\eta(b) $, which implies
$ \eta (a) \sqsubseteq \eta(b) $.
 
Notice that, even when $\mathbf S$ and $\mathbf T$
are principal, a specialization homomorphism
need not be a $K$-homomorphism;
see, for example, the second sentence
 in the following Remark \ref{top}. 
Of course, if either $\mathbf S$ or $\mathbf T$
fails to be principal, then it is not even possible to apply
the notion of $K$-homomorphism.
Whenever we speak of a homomorphism without
further specifications, we always mean a homomorphism
of specialization semilattices as introduced above, that is,
we do not assume that homomorphisms are $K$-homomorphisms,
unless specified otherwise.

A principal specialization semilattice
(or a closure semilattice) is \emph{additive}
if $K(a \vee b)= Ka \vee Kb$.   

\begin{remark} \labbel{top}
If $X$ is a topological space with topological closure $K$,
then $(\mathcal P, \cup, K)$ is an additive closure semilattice,
thus  $(\mathcal P, \cup, \sqsubseteq )$ is a 
principal additive specialization semilattice,
by Remark \ref{psc}(a).

It can be checked that topological continuity corresponds to
the notion of  homomorphism
 between the associated specialization semilattices
 \cite[Proposition 2.4]{mtt}; on the other hand, the notion of
$K$-homomorphism is stronger, and corresponds to the notion of
a closed continuous map.   

All the above comments apply to \emph{closure spaces},
which are like topological spaces, except that the union of two
closed subsets is not assumed to be closed, equivalently, closure
is not assumed to satisfy $K(a \cup b) \subseteq Ka \cup Kb$.
In a closure space the closure of the empty set is not assumed
to be the empty set, either.
Closure spaces occur naturally in algebra; for example, if $\mathbf G$
is a group, then $\mathcal P(G)$ becomes a closure space if 
subgroups are considered as the closed subsets of $G$.
See \cite{E} for more examples and details.
Of course,  in the case of a closure space, the associated
specialization semilattice as above is still principal, 
 but not necessarily additive.
\end{remark}   

Further details about  the above  notions can be found in \cite{mtt}.

A \emph{specialization semilattice with $0$} 
is a specialization semilattice with a constant $0$
which is a neutral element with respect  to the semilattice operation,
thus a minimal element in the induced order,
and furthermore satisfies  
\begin{equation}  
\labbel{s0}    \tag{S0}
 a \sqsubseteq 0 
  \Rightarrow 
 a =0. 
 \end{equation} 
A homomorphism $\eta$ of specialization semilattices with
$0$ is required to satisfy $\eta(0)=0$. 
When some risk of ambiguity might occur,
we shall explicitly mention that the homomorphism
is \emph{$0$-preserving}.   

\begin{remark} \labbel{0}
   We shall generally assume that specialization semilattices
have a $0$, but this assumption is only for simplicity.
In fact, if $\mathbf S$  is an arbitrary specialization semilattice,
then by adding a new $\vee$-neutral element $0$
and setting $0 \sqsubseteq a $, for every     
 $a \in S \cup \{ 0 \} $,
and $ a \not\sqsubseteq 0 $, for every   
$a \in S$, we get a specialization semilattice with $0$.  
Conversely, if $\mathbf S$ is a specialization semilattice 
with $0$, then $S \setminus \{ 0 \} $
has naturally the structure of a specialization semilattice.  
 \end{remark}

\section{Universal extensions} \labbel{univ1}

Given any specialization semilattice $\mathbf S$,
we now construct a ``universal'' principal 
additive extension $ \widetilde {  \mathbf S}$  of $\mathbf S$.

\begin{definition} \labbel{und}   
Suppose that $\mathbf S$ is a specialization semilattice.

On the product 
$S \times S$ define an equivalence relation 
$ \sim$ by
  \begin{enumerate}   
 \item[(*)]
$ (a,b ) \sim (c,d)$ 
if and only if,   in $\mathbf S$,
$b \sqsubseteq d$, $d \sqsubseteq b$ 
and there are  $a_1, c_1 \in S$  
such that  $ a_1 \sqsubseteq b$,
$ c_1 \sqsubseteq d$
and
$ a  \leq c \vee c_1$,
 $ c \leq  a \vee a_1$.
   \end{enumerate} 
We shall check in Lemma \ref{corre}(i)
below that $ \sim$ is actually an equivalence relation.  
Let 
$ \widetilde {   S} = 
( S \times  S) /{\sim}$. 

Define $K: \widetilde {   S} \to \widetilde {   S}$
by $K[a,b] = [a, a \vee b]$,   
where $[x,y]$
is the $ \sim$ class of the pair $(x,y)$.    
In Lemma \ref{corre}(ii)(iii) we shall  prove that $K$ is well-defined
and that  $\widetilde {   S}$ naturally inherits
a semilattice operation $  \vee$ from 
the semilattice product
$\mathbf S \times\mathbf S$.

Define $ \sqsubseteq $ on   $\widetilde {   S}$ by
$ [a,b] \sqsubseteq [c,d]$
if  $ [a,b] \leq K[c,d]$, where $\leq$ 
is the order induced by $   \vee$
and let $\widetilde {   \mathbf S} = (\widetilde { S},   \vee, \sqsubseteq ) $,
$\widetilde {   \mathbf S}' = (\widetilde { S},   \vee, K) $.

If $\mathbf S$ is a specialization semilattice with $0$, 
 define
$\upsilon_{_\mathbf S} :  S \to \widetilde { S}$ 
by 
$\upsilon_{_\mathbf S} (a)= [a,0]$.
\end{definition}

We intuitively think
of  $[a,b]$ as $a \vee Kb$, where $Kb$ 
is the ``new'' closure we need to introduce;
in particular,  $[a,0]$ corresponds to $a$
and $[0,b]$ corresponds to a  new element $Kb$.

\begin{theorem} \labbel{univt}
Suppose that  $\mathbf S$ is a specialization semilattice
with $0$. Let $ \widetilde {  \mathbf S}$ and
$\upsilon_{_\mathbf S} $ be  as in Definition \ref{und}.
Then the following statements hold. 
  \begin{enumerate}   
 \item 
$ \widetilde {  \mathbf S}$ is 
a principal 
additive specialization semilattice with $0$. 
\item 
$\upsilon_{_\mathbf S} $ is a 
$0$-preserving 
 embedding of $\mathbf S$ 
 into $\widetilde {  \mathbf S}$.
\item  
The pair $ (\widetilde {  \mathbf S}, \upsilon_{_\mathbf S})$
has the following universal property.

For every  principal 
additive specialization semilattice
$\mathbf  T$  
 and every homomorphism
 $ \eta : \mathbf S \to \mathbf  T$, there
is a unique  $K$-{\hspace{0 pt}}homomorphism  
 $ \widetilde{\eta} : \widetilde{\mathbf S} \to \mathbf  T$
such that
$\eta = \upsilon_{_\mathbf S} \circ \widetilde{\eta}$. 
\begin{equation*}
\xymatrix{
	{\mathbf S}  \ar[rd]_{\eta}
 \ar[r]^{\upsilon_{_\mathbf S}}
 &\widetilde{\mathbf  S} \ar[d]^{\widetilde{\eta}}\\
	 &{\mathbf  T}}
   \end{equation*}    
\item
If $\mathbf U$
is another specialization semilattice with $0$  
and $\psi : \mathbf S \to \mathbf U$ 
is a 
$0$-preserving 
homomorphism, then 
there is a unique 
 $K$-homomorphism
$\widetilde{\psi}: \widetilde{\mathbf S} \to \widetilde{\mathbf U}$
making the following diagram commute.
\begin{equation*}
\xymatrix{
	{\mathbf S} \ar[d]_{\psi} \ar[r]^{\upsilon_{_\mathbf S}}
 &\widetilde{\mathbf S} \ar[d]^{\widetilde{\psi}}\\
	{\mathbf U} \ar[r]^{\upsilon_{_\mathbf U}}
 &{\widetilde{\mathbf U}}}
   \end{equation*}    
 \end{enumerate} 
 \end{theorem}

We first need  to 
check that Definition \ref{und}
is correct. This is the content
of the following lemma.

\begin{lemma} \labbel{corre}
Assume the notation and the definitions in \ref{und}. 
  \begin{enumerate}[(i)]    
\item 
The relation $ \sim$ on $S \times S$ is an equivalence relation. 
\item
The operation $K$ is well-defined on the $ \sim$-equivalence classes.
\item
The relation $ \sim$ is a semilattice congruence
on the semilattice product $\mathbf S \times \mathbf S$,
hence $\widetilde{\mathbf S}$
inherits  a semilattice structure from
$\mathbf S \times \mathbf S$.
\item
If $\mathbf S$ is a specialization semilattice with $0$, then
$K$ satisfies
\begin{equation*}   
K[a,b] = [0, a \vee b]= [a, a \vee b ].
  \end{equation*}
 \end{enumerate} 
 
 \end{lemma} 

\begin{proof}    
(i) The relation $ \sim$ is 
symmetric, since its definition
is symmetric. 
It is reflexive 
because of \eqref{s4}, since if
$a=c$ and  $b=d$, then we can take
$a_1=c_1=b$ in (*).   
We now check transitivity.
Let $ (a,b ) \sim (c,d)$ be witnessed by elements
 $a_1, c_1 $ as in (*).
Let
$ (c,d) \sim (e,f)$ be witnessed by 
 $c'_1, e'_1 $, thus
$d \sqsubseteq f$, $f \sqsubseteq d$, 
  $ c'_1 \sqsubseteq d$,
$ e'_1 \sqsubseteq f$
and
$ c  \leq e \vee e'_1$,
 $ e \leq  c \vee c'_1$.
Then $b \sqsubseteq f$,
 by   $ b\sqsubseteq d$, $d \sqsubseteq f$ and
\eqref{s2}. 
Symmetrically $ f \sqsubseteq b$.
From $ a  \leq c \vee c_1$
and  $ c \leq  e \vee e'_1$
we get $ a  \leq e \vee e'_1 \vee c_1$.
Moreover, 
$ c_1 \sqsubseteq d \sqsubseteq f $,
hence 
$ c_1  \sqsubseteq f $
by \eqref{s2},
thus
$e'_1 \vee c_1 \sqsubseteq f$,
by    $ e'_1 \sqsubseteq f$
 and \eqref{s3}.
Symmetrically,
$ e  \leq a \vee a_1 \vee c'_1$
and
$a_1 \vee c'_1 \sqsubseteq b$. 
This means that the elements
$ a''_1 =  a_1 \vee c'_1$
and
$ e''_1 =  e'_1 \vee c_1$
witness $ (a,b) \sim (e,f)$.

(ii) We have to show that 
if $ (a,b) \sim (c,d)$,
then $ (a, a \vee b) \sim (c, c \vee d)$.
Suppose that $ (a,b) \sim (c,d)$ is witnessed by
$a_1$ and  $c_1$, as given by (*) in Definition \ref{und}. 
 From $a \leq c \vee c_1$ and 
$ c_1 \sqsubseteq d$ 
 we get $ a \sqsubseteq  c \vee c_1 
\sqsubseteq c \vee d$, by \eqref{s1}, \eqref{s4} and \eqref{s7},
hence $ a \sqsubseteq  c \vee d$, by \eqref{s2}. 
Since $ b \sqsubseteq d \leq c \vee d$,
hence $ b \sqsubseteq c \vee d$, by \eqref{s1} and \eqref{s2}, then
we also have $ a \vee b \sqsubseteq c \vee d$,
by \eqref{s3}. 
Symmetrically,
$ c \vee d  \sqsubseteq a \vee b$.
The remaining conditions in Clause (*)
in Definition \ref{und} are  verified
by using the same $a_1 $ and $ c_1$.
Indeed, from $a_1 \sqsubseteq b$ and $b \leq a \vee b$,
hence $b \sqsubseteq  a \vee b$ by \eqref{s1}, we get
$a_1 \sqsubseteq   a \vee b$, by \eqref{s2}. Similarly,
$c_1 \sqsubseteq   c \vee d$.
The last two conditions in (*) hold
since these conditions do not involve $b$ and $d$,
and $a$ and  $c$ have remained unchanged. 

Hence 
$ (a, a \vee b) \sim (c, c \vee d)$.
This means that $K$ is well-defined.

(iii)
We have to show that if  $ (a,b) \sim (c,d)$,
then  $ (a,b) \vee (e,f) \sim (c,d) \vee (e,f) $,
that is, $ (a \vee e ,b \vee f) \sim (c \vee e,d \vee f)$.
Since $ (a,b) \sim (c,d)$,
then $ b \sqsubseteq d$, hence  $ b \vee f \sqsubseteq d \vee f$
follows from \eqref{s4} and  \eqref{s7}.
Symmetrically,
$ d \vee f  \sqsubseteq b \vee f$.
Again by $ (a,b) \sim (c,d)$,
there is $c_1 \sqsubseteq d$
such that $a \leq c \vee c_1$.   
Then 
$c_1 \sqsubseteq d \vee f$
by  \eqref{s2} (since $d \sqsubseteq d \vee f$ 
by \eqref{s1}); moreover,  
$a \vee e \leq c \vee e \vee c_1$.
Performing the symmetrical argument,
we get $a_1 \sqsubseteq b \vee f$
and $c \vee e \leq a \vee e \vee a_1$. This means 
that the same elements 
$c_1$ and $a_1$ witnessing   
$ (a,b) \sim (c,d)$ also
witness 
$ (a \vee e ,b \vee f) \sim (c \vee e,d \vee f)$.
We have shown that 
$ \sim$ is a semilattice congruence.

(iv) Since  we have shown that $K$ 
is well defined on the equivalence classes,
 in order to get the equation in (iv) it is enough to check that  
if $\mathbf S$ has a $0$, then  
$ (a, a \vee b) \sim (0, a \vee b)$.
The condition $ a \vee b \sqsubseteq a \vee b $
is from \eqref{s4}. Then take 
$c_1 =  a \vee b$ and $a_1=0$ 
in order to witness (*). 
 \end{proof}

 \begin{proof}[Proof of Theorem \ref{univt}]
Definition \ref{und} is justified
by Lemma \ref{corre}.
In order to prove Clause (1) in the theorem 
it is easier to deal with the structure $\widetilde{\mathbf S}'$
from Definition \ref{und}.

\begin{claim*} \labbel{cl} 
$\widetilde {   \mathbf S}' = (\widetilde { S},   \vee, K) $
is  an additive closure semilattice.
 \end{claim*}  

We have  shown in Lemma \ref{corre}(iii)  that 
$(\widetilde { S},   \vee)$ is a semilattice; it remains to check
that $K$ is an additive closure. 
Indeed, by the definition of $K$,
\begin{align*} \labbel{alk}    
& [a,b]  \leq [a,a \vee b] = K[a,b],
\\
& KK[a,b]= K[a,a \vee b]=[a,a \vee a \vee b]=K[a,b], \text{ and }
\\
& K([a,b] \vee [c,d]) = K[a \vee c, b \vee d] = [a \vee c, a \vee b \vee c \vee d]
\\
& \phantom{K([a,b] \vee [c,d])}=
[a, a \vee b] \vee [c,c \vee d]=  K[a,b] \vee K[c,d].
 \end{align*} 

Having proved the claim, Clause (1)
in the theorem follows 
from Remark \ref{psc}(a).
The $0$ of  $\widetilde {   \mathbf S}$
is the class $[0,0]$, since $(0,0)$ is a neutral element for
$\mathbf S \times \mathbf S$, hence 
$[0,0]$ is neutral for the quotient $\widetilde {   \mathbf S}=
   \mathbf S/ { \sim } $.
Moreover, $[a,b] \sqsubseteq [0,0]$
means $[a,b] \leq K[0,0]=[0,0]$
and this implies   $[a,b] =[0,0]$, since
$[0,0]$ is the minimum of $\widetilde {   \mathbf S}$.

Now we prove (2).
We have
$\upsilon_{_\mathbf S} (a \vee b) = [a \vee b, 0]
= [a , 0] \vee [b,0] = \upsilon_{_\mathbf S} (a) \vee \upsilon_{_\mathbf S} ( b)$,
hence $\upsilon_{_\mathbf S} $ is a semilattice homomorphism.
Moreover, $\upsilon_{_\mathbf S} $ is injective, since
$\upsilon_{_\mathbf S} (a) =\upsilon_{_\mathbf S} ( c 
)$ 
means 
$(a,0) \sim (c
,0)$
and this happens only if 
$a \leq c$ and $c \leq a$,
that is,  $a=c$.
Indeed, if    $b=d=0$ and  $ a_1 \sqsubseteq b$,
$ c_1 \sqsubseteq d$   in Clause (*) in  Definition \ref{und},
then $a_1=c_1=0$ by \eqref{s0}.    

Furthermore, 
if $ a \sqsubseteq b $ in $\mathbf S$,
then 
$ a \vee b \sqsubseteq  b $,
by \eqref{s4} and
  \eqref{s3}.
We then get
$(a \vee b, b) \sim (b,b)$,
by \eqref{s4} and  taking  
$c_1=a \vee b$ and
$a_1=b$ in (*) from  
Definition \ref{und}. 
Hence $[a,0] \leq [a \vee b, b] = [b,b] = K[b,0]$,  
that is,  $\upsilon_{_\mathbf S} (a) \sqsubseteq \upsilon_{_\mathbf S} (b)$,
according to the definition  of  $ \sqsubseteq $  on
 $\widetilde{S}$  in Definition \ref{und}.
This shows that $\upsilon_{_\mathbf S}$
is a $ \sqsubseteq $-homomorphism.

In fact, $\upsilon_{_\mathbf S}$ is an embedding, since from
$\upsilon_{_\mathbf S} (a) \sqsubseteq \upsilon_{_\mathbf S} (b)$,
that is, 
$[a,0] \leq  K[b,0] =  [b,b] $,
we get 
$[a \vee b,b]= [a,0] \vee [b, b] =  [b,b] $,
that is, $(a \vee b ,b) \sim (b,b) $,
hence, according to (*) in Definition \ref{und},
 $a \vee b \leq b \vee c_1$,  for some $c_1 \sqsubseteq b$.  
From $ a \leq a \vee b \leq b \vee c_1$
and \eqref{s1} we get $a \sqsubseteq b \vee c_1$.
From  $c_1 \sqsubseteq b$, \eqref{s4} and  \eqref{s3}
we get  $ b \vee c_1 \sqsubseteq b$,
hence $a \sqsubseteq b$ by \eqref{s2}.
This shows that $\upsilon_{_\mathbf S}$ is an embedding.
 
Since $\upsilon_{_\mathbf S}(0)=[0,0]$,
then $\upsilon_{_\mathbf S}$ is $0$-preserving.

We now deal with (3).
If $\eta: \mathbf S \to \mathbf T$
is a homomorphism and there exists 
$\widetilde{\eta}$ 
such that   
$\eta = \upsilon_{_\mathbf S} \circ \widetilde{\eta}$,
then
$\widetilde{\eta}([a,0]) =
   \widetilde{\eta}(\upsilon_{_\mathbf S}(a)) = \eta (a) $, for every $a \in S$.
If furthermore $ \widetilde{\eta}$
is a $K$-homomorphism, then 
$ \widetilde{\eta}([0,b]) = ^{\ref{corre}}  \widetilde{\eta}(K[b,0])=
K \widetilde{\eta}([b,0]) = K \eta (b)$, by
the equation  in Lemma \ref{corre}(iv).
Since $\widetilde{\eta}$ is supposed to be
a lattice homomorphism, 
it follows that 
$\widetilde{\eta}([a,b])
=\widetilde{\eta}([a,0]) \vee \widetilde{\eta}([0,b])=
 \eta (a) \vee  K\eta (b)$, 
hence if $   \widetilde{\eta}$  
exists it is unique.
The above considerations make sense since 
$\mathbf T$ is assumed to be principal, so that 
 $K \eta (b)$ exists. 

It is then enough to show that the above condition
$\widetilde{\eta}([a,b]) =  \eta (a) \vee  K\eta (b)$ 
actually determines a $K$-homomorphism
$   \widetilde{\eta}$ from 
$\widetilde{ \mathbf S}$  
to $\mathbf T$.

First, we need to check that 
if $ (a,b) \sim (c,d)$,
then  $ \eta (a) \vee  K\eta (b) =\eta(c) \vee K\eta(d) $,
so that $\widetilde{\eta}$
is well-defined.
In fact, 
suppose that $ (a,b) \sim (c,d)$ is given
by (*) in \ref{und}. 
From
$b \sqsubseteq d$
and  $ d \sqsubseteq b$,
we get
$ \eta (b)\sqsubseteq \eta(d)$
and  $ \eta(d) \sqsubseteq \eta(b)$,
since $\eta$ is a homomorphism, 
 so that $K\eta(b)= K\eta(d)$,
by the final sentence in Remark \ref{psc}(c).
Moreover, 
if 
 $ c_1 \sqsubseteq d$,
then 
 $ \eta(c_1) \sqsubseteq \eta(d)$,
so that 
 $ \eta(c_1) \leq K\eta(d)$,
by the definition of $K$. 
Since in addition
$ a  \leq c \vee c_1$,
then 
$ \eta(a)  \leq 
\eta(c \vee c_1) =
\eta(c) \vee \eta(c_1) 
\leq \eta(c) \vee K\eta(d) $,
since $\eta$ is a homomorphism, 
so that 
$ \eta(a)  \vee K\eta(b) 
\leq \eta(c) \vee K\eta(d) $,
since we have already shown that 
$K\eta(b)= K\eta(d)$.
Symmetrically,  
$ \eta(c) \vee K\eta(d)  
\leq \eta(a)  \vee K\eta(b)$,
hence $ \eta(c) \vee K\eta(d)  
= \eta(a)  \vee K\eta(b)$.
This means that $   \widetilde{\eta}$ is well-defined.

We now check that 
$   \widetilde{\eta}$ is a
semilattice homomorphism. 
Indeed, 
\begin{align*} 
 \widetilde{\eta}([a,b])  \vee \widetilde{\eta}([c,d]) &=
  \eta (a) \vee  K\eta (b) \vee   \eta (c) \vee  K\eta (d)
\\
 &=
  \eta (a) \vee    \eta (c) \vee  K\eta (b) \vee  K\eta (d)
\\
 &=^{ \text{A}}
\eta( a \vee c)  \vee  K(\eta (b) \vee  \eta (d))
\\
 &=
\eta( a \vee c)  \vee  K\eta (b \vee  d) =
\widetilde{\eta}([a \vee c ,b \vee d]) ,
 \end{align*}
where we have used the definition of
$\widetilde{\eta}$, the  assumption
that $\eta$ is a semilattice homomorphism and 
in the identity marked with the superscript $A$ 
we have used the assumption that $\mathbf T$  
is additive.

Finally, $   \widetilde{\eta}$ is a
$K$-homomorphism,
since
\begin{multline*} 
       \widetilde{\eta}(K[a,b])=
 \widetilde{\eta}([a, a \vee b])=
\eta(a) \vee K\eta(a \vee b) = ^{\diamondsuit} 
K\eta(a \vee b) =
\\
K(\eta(a) \vee \eta(b))= ^{\ref{psc}} 
K (\eta(a) \vee K\eta(b))=
K \widetilde{\eta}([a,b]),
\end{multline*}   
where we have used 
the definitions of $K$ and $\widetilde{\eta}$,
the assumption that $\eta$ is a homomorphism of specialization semilattices 
with $0$ and
Remark \ref{psc}(b). The equation marked with $\diamondsuit$
follows from $\eta(a) \leq  \eta(a \vee b) \leq  K\eta(a \vee b)$. 

Notice that we do not need to assume that 
$\mathbf T$ has a $0$, in order to get (3).
However, if $\mathbf T$ has a $0$ and
$\eta$ is $0$-preserving,  then $\widetilde{\eta}$
is  $0$-preserving, too, since we have proved that 
$\upsilon_{_\mathbf S}$ is $0$-preserving and
that  the diagram in (3) commutes.

Having proved Clause (3) in the theorem, 
we now prove Clause (4).
Suppose that  $\mathbf U$
is a specialization semilattice with $0$ 
and $\psi : \mathbf S \to \mathbf U$ 
is a $0$-preserving homomorphism. 
Then  $\eta = \psi \circ \upsilon_{_{\mathbf U} }$
is a $0$-preserving homomorphism from 
$\mathbf S$ to  $ \widetilde{\mathbf  U}$,
by clauses (1) and (2) applied to $\mathbf U$. 
Applying Clause 
 (3) 
to $\eta $
and $\mathbf  T=  \widetilde{\mathbf  U}$,
and by the last comment in the proof of (3),
we get that there is a unique $0$-preserving $K$-homomorphism
 $ \widetilde{ \eta } $ from 
$\widetilde{\mathbf S}$ to $\widetilde{\mathbf  U}$
such that  
 $ \eta = \upsilon_{_{\mathbf S} } \circ \widetilde{ \eta } $.
Letting $\widetilde{ \psi } = \widetilde{ \eta }$,
then the diagram in (4)  commutes, since 
we have taken $\eta = \psi \circ \upsilon_{_{\mathbf U} }$.
Conversely, if  $\widetilde{ \psi }$ 
is a $K$-homomorphism
which makes the diagram 
in (4) commute, then necessarily 
$\widetilde{ \psi } = \widetilde{ \eta }$, because of the unicity
of $\widetilde{ \eta }$ in (3).
 
 \end{proof}

Notice that 
$\upsilon_{_\mathbf S}$ as given by Theorem \ref{univt}(2) does not necessarily 
preserve existing closures in $\mathbf S$:
just consider the case in which $\mathbf S$
is principal but not additive,
then closures necessarily are modified, since
$\widetilde{\mathbf S}$
turns out to be additive.   

Moreover, it is necessary to ask that 
 $\widetilde{\eta}$ is 
a $K$-homomorphism in Theorem \ref{univt}(3);
 it is not enough to assume that
 $\widetilde{\eta}$ is 
just a homomorphism.
Indeed, let $\mathbf S = \mathbb N$ with 
$\max$ as join and with $ n \sqsubseteq m $,
for    all $m,n >0$. 
Then $\widetilde{ \mathbf S}$
is isomorphic  to 
$\mathbf S \cup \{  \infty\} $,
where $Ka= \infty$,
for every   $a \in S \cup \{  \infty\}$,
$a \neq 0$.
Let $T= \{ 0,1, 2 \} $
with $ 2 \sqsubseteq 1$
and with  the standard interpretation otherwise.
Let $\eta: \mathbf S \to \mathbf T$
with $\eta(0)=0$
and $\eta(n)=1$ otherwise.  
Then   the only $K$-homomorphism
extending $\eta$ 
must send $\infty$ to $2=K(1)$.
However, if we set
$\eta^*(\infty) =1 $,
we still get a (not $K$-) homomorphism
from $\widetilde{ \mathbf S}$ to $\mathbf T$ 
extending $\eta$.

\begin{remark} \labbel{w0}
For simplicity, we have stated
and proved  Theorem \ref{univt}
for specialization semilattices with $0$,
but the theorem holds  for arbitrary  specialization semilattices.
  
If $\mathbf S_1$ does not have a
$0$, first apply the theorem to
  $\mathbf S=\mathbf S_1 \cup \{ 0 \} $
as constructed in Remark \ref{0}
and then restrict to    $\mathbf S_1$
and $ \widetilde{\mathbf S} \setminus \{ 0 \} $.
Notice that  $\upsilon_{_\mathbf S}$
sends $0$ to  $0$. 

In order to prove (3),
if $\eta_1: \mathbf S_1 \to \mathbf T_1$,
add a new $0$ to $\mathbf T_1$, getting some 
specialization semilattice $\mathbf T$. 
Extend $\eta_1$ to some homomorphism
$\eta: \mathbf S \to \mathbf T$
 by setting 
$\eta(0)=0$.
Having obtained (3) in the extended situation,
it follows that (3) holds for the original
 $\eta_1 $, $  \mathbf S_1 $ and $  \mathbf T_1$.
\end{remark}

\section{More general universal extensions} \labbel{univ}

In the present section we assume that the reader
is familiar with some basic notions of model theory \cite{CK}.
The following lemma about the existence of 
universal objects is a  folklore argument.
A \emph{subreduct} is a substructure of some reduct. 

In the next lemma $\mathscr L \subseteq \mathscr L'$ 
are two languages, $\mathcal K'$ is a class of models
for $\mathscr L'$  and $\mathcal K$ is the class of all
subreducts in the language $\mathscr L$ of members
of $\mathcal K'$.
We adopt the convention that models in 
$\mathcal K'$ are denoted by 
$\mathbf A'$, $\mathbf  B'$, \dots \  and 
$\mathbf A$, $\mathbf  B$, \dots
\ are the corresponding $\mathscr L$-reducts.

\begin{lemma} \labbel{exp}
Under the above assumptions, if 
$\mathcal K'$ is closed under isomorphism,
substructures and products,
 then, for every $\mathbf A \in \mathcal K$,
there are $ \widetilde{\mathbf  A}' \in \mathcal K'$ 
and an $\mathscr L$-embedding 
$\upsilon_{_\mathbf A}  : \mathbf A \to \widetilde{\mathbf  A}$
 such that, for every 
$\mathbf  B' \in \mathcal K'$  
 and $\mathscr L$-homomorphism $ \eta : \mathbf A \to \mathbf  B$, there
is a unique  $\mathscr L'$-homomorphism  
 $ \widetilde{\eta} : \widetilde{\mathbf A}' \to \mathbf  B'$
such that  $\eta = \upsilon_{_\mathbf A} \circ \widetilde{\eta}$. 
\begin{equation*}
\xymatrix{
	{\mathbf A}  \ar[rd]_{\eta}
 \ar[r]^{\upsilon_{_\mathbf A}}
 &\widetilde{\mathbf  A} \ar[d]^{\widetilde{\eta}}
&\widetilde{\mathbf A}' \ar[d]^{\widetilde{\eta}}\\
	 &{\mathbf  B}&{\mathbf B}'
}
   \end{equation*}    
 
The structure $\widetilde{\mathbf A}'$
is unique up to isomorphism over $\upsilon_{_\mathbf A}(A)$. 
As a consequence, if $\mathbf E \in \mathcal K$
and $\psi : \mathbf A \to \mathbf E$ 
is an $\mathscr L$-homomorphism, then 
$\psi$ lifts to an $\mathscr L'$-homomorphism
$\widetilde{\psi}: \widetilde{\mathbf A}' \to \widetilde{\mathbf E}'$
making the following diagram commute.
\begin{equation*}
\xymatrix{
	{\mathbf A}  \ar[d]_{\psi}
 \ar[r]^{\upsilon_{_\mathbf A}}
 &\widetilde{\mathbf  A} \ar[d]^{\widetilde{\psi}}
&\widetilde{\mathbf A}' \ar[d]^{\widetilde{\psi}}\\
	{\mathbf E} \ar[r]^{\upsilon_{_\mathbf E}}&\widetilde{\mathbf  E}&\widetilde{\mathbf  E}'
}
   \end{equation*}    
\end{lemma}

 \begin{proof} 
The proof is a standard construction of free objects.
Since $\mathbf A \in \mathcal K$,
then $\mathbf A$ is a subreduct of some $\mathbf C' \in \mathcal K'$.
Since $\mathcal K'$ is closed under substructures,
 we can choose $\mathbf  C'$ in such a way that
$\mathbf  C'$ is generated by $A$ in the language $\mathscr L'$.
Consider the class of all $\mathbf  C' \in \mathcal K'$
such that there is a homomorphism $\xi$ from $\mathbf A$ to 
$\mathbf  C$ and $\mathbf  C'$ is generated by    
$\xi(A)$ in the language $\mathscr L'$;
by the preceding sentence this class is nonempty.  
If $\xi, \mathbf  C'$ and $\xi_1, \mathbf  C'_1$
are as above, let us call the pairs $(\xi, \mathbf  C')$ and $(\xi_1, \mathbf  C'_1)$
\emph{equivalent} if there is an isomorphism 
$ \psi :\mathbf  C'_1 \to  \mathbf  C'$ such that 
$ \xi  = \xi_1 \circ \psi $, namely,
$\mathbf  C' $ and $   \mathbf  C'_1$ are isomorphic over
the image of $A$. 
Let $( \mathbf C'_i, \xi_i) _{i \in I} $ 
be a family of representatives for each 
equivalence class. 
Since each $\mathbf  C'$
as above is generated by  $\xi(A)$ in the language $\mathscr L'$,
we have $|C'| \leq \sup \{ \omega, |A|, |\mathscr L'| \} $,
hence any representative can be taken over some fixed set
of cardinality $\sup \{ \omega, |A|, |\mathscr L'| \} $;
in conclusion, there is a set---not a proper class---of
such representatives. 

Let $\mathbf  D' = \prod _{i \in I} \mathbf C'_i $,
thus $\mathbf  D' \in \mathcal K'$, since $\mathcal K'$  
is closed under products.
Let 
$ \widetilde{\mathbf  A}'$
be the substructure of $\mathbf  D'$ 
$\mathscr L'$-generated by the sequences 
 $ (\xi _i (a)) _{i \in I}$, for $a $ varying  in $A$.
Since $\mathcal K'$ is closed 
under substructures, then  
$ \widetilde{\mathbf  A}' \in \mathcal K'$.
Moreover, the function which assigns to $a \in A$
 the sequence $ (\xi _i (a)) _{i \in I}$
is an $\mathscr L$-embedding $\upsilon_{_\mathbf A}$  from $\mathbf A$ to
 $ \widetilde{\mathbf  A}$;  $\upsilon_{_\mathbf A}$
 is an embedding because of the 
first sentence in the proof.
Notice that, for every $i \in I$, the projection $\pi_i$ from
$\mathbf  D'$ to $\mathbf C'_i $ 
induces a homomorphism
$\zeta_i : \widetilde{\mathbf  A}' \to \mathbf C'_i$
such that $\xi_i=\upsilon_{_\mathbf A} \circ \zeta _i$. 

If
$\mathbf  B' \in \mathcal K'$  
 and $ \eta : \mathbf A \to \mathbf  B$
is a  homomorphism, 
let  
$\mathbf  B'_1$ be the $\mathscr L'$-substructure
of $\mathbf  B'$ generated by $\eta (A)$,
let $\iota$ be the inclusion embedding from  
$\mathbf  B'_1$ to $\mathbf  B'$ and let $\eta_1$ be
the function  induced by 
$\eta$ from $\mathbf A$ to $\mathbf  B_1$,  that is,
$\eta= \eta _1 \circ \iota$.
By the choice of the $\mathbf  C'_i$s, $\mathbf  B'_1$ 
  is isomorphic  
to  $\mathbf C'_i$, for some $i \in I$,
through an isomorphism $\psi$
such that  $\eta_1  = \xi_i \circ \psi$, hence 
$\eta= \eta _1 \circ \iota = \xi_i \circ \psi \circ \iota  = 
\upsilon_{_\mathbf A} \circ \zeta _i \circ \psi \circ \iota$. 
It follows that  $\widetilde{\eta} = \zeta_i  \circ \psi \circ   \iota$    
is the desired homomorphism. 

Since,
by construction,  
$ \widetilde{\mathbf  A}'$
is 
$\mathscr L'$-generated by $\upsilon_{_\mathbf A}(A)$,
then every element  of $ \widetilde{\mathbf  A}'$ has the form
$t(\upsilon_{_\mathbf A}(a_1), \dots, \upsilon_{_\mathbf A}(a_n))$,
for some natural number $n$, some term
$t$ of $\mathscr L'$   and elements 
$a_1, \dots, a_n$ of $A$.   
The requests that $\widetilde{\eta}$ be
an $\mathscr L'$-homomorphism and that 
$\eta = \upsilon_{_\mathbf A} \circ \widetilde{\eta}$ imply
that 
\begin{align*}
 \widetilde{\eta}(t(\upsilon_{_\mathbf A}(a_1), \dots, \upsilon_{_\mathbf A}(a_n)))
&= t(\widetilde{\eta}(\upsilon_{_\mathbf A}(a_1)), \dots, 
\widetilde{\eta}(\upsilon_{_\mathbf A}(a_n)))
\\
&= t(\eta(a_1), \dots, \eta(a_n)),
\end{align*}   
 hence $\widetilde{\eta}$ is uniquely determined.

To prove the last statement, just 
take $\eta = \psi \circ \upsilon_{\mathbf E }$,
$\mathbf  B=  \widetilde{\mathbf  E}$
and argue as in the proof of clause (4)
in Theorem \ref{univt}. 
\end{proof}

In particular, Lemma \ref{exp}
applies when 
  $\mathcal K'$ is the class of the models
of some universal Horn 
first-order theory $T'$  
in the language $\mathscr L'$.

 Lemma \ref{exp},
together with the above comment,
can be applied in all the
situations described below.

  \begin{enumerate}[(C1)]
\item
$\mathscr L'$ is the language of Boolean algebras plus
a binary relation symbol  $ \sqsubseteq $ 
and a unary operation symbol $K$.
 $T'$ is the theory of \emph{closure algebras},
that is, $T'$ contains the axioms for Boolean algebras plus
 axioms saying that $K0=0$ and $K$ is 
extensive, idempotent and additive and
let us add to $T'$ an axiom defining  $ \sqsubseteq $,
namely,
$ a \sqsubseteq b \Leftrightarrow a \leq Kb   $. 

Finally, $\mathscr L= \{ \vee, \sqsubseteq  \} $.  

\item
 $\mathscr L'$ is the language of
closure semilattices  plus
 a binary relation symbol  $ \sqsubseteq $. 
$T'$ is the theory of closure semilattices
plus axioms defining  $ \sqsubseteq $, as above,
 $\mathscr L= \{ \vee, \sqsubseteq  \} $.  

\item
As in (C1), but $K$ is only assumed to be 
extensive, idempotent and isotone.

\item
As in (C2), plus the assumption that $K$ is additive.

\item
As in (C2), plus the assumption that 
$K$ satisfies $a \vee Kb= K(a \vee b)$.

\item
 $\mathscr L'$ is the language of
closure posets  plus
 a binary relation symbol  $ \sqsubseteq $. 
$T'$ is the theory of closure posets
plus axioms defining  $ \sqsubseteq $.
Let $\mathscr L= \{ \leq , \sqsubseteq  \} $.  

\item
We can allow 
$\mathscr L= \{ \leq , \sqsubseteq  \} $ 
also in all cases (C1)-(C5), adding 
the symbol $\leq$ to $\mathscr L'$, with its definition 
 $a \leq b \Leftrightarrow a\vee b = b $. 
    \end{enumerate}   

Recall that $\mathcal K$ is the class
of all $\mathscr L$-subreducts of models of $\mathcal K'$. 
In  cases (C1)-(C5) the class 
$\mathcal K$ turns out to be the class of all specialization semilattices,
since we have proved in \cite[Theorem 4.10]{mtt}
that every specialization semilattice can be embedded into
the specialization semilattice associated to some topological
space $X$. In particular, this provides an embedding
into the  specialization closure algebra 
$(P(X), \cap, \cup, \complement, \emptyset, X, K, \sqsubseteq )$;
for cases (C2) - (C4) it is then sufficient to consider an appropriate reduct.

For case (C5), it follows from the proof of  \cite[Theorem 4.8]{mtt} 
 that every specialization semilattice can be extended to some principal
specialization semilattice satisfying 
$a \vee Kb= K(a \vee b)$.
In fact, for case (C5) the construction in the proof 
\cite[Theorem 4.8]{mtt} provides an explicit description 
for the universal object whose existence follows
from Lemma \ref{exp}.
Notice also that Theorem \ref{univt} here
provides a description for the universal 
object corresponding to (C4).  

In cases (C6) and (C7)
the class $K$ is the class of specialization posets,
since we have shown in \cite[Proposition 4.15]{mtt} 
that every specialization poset can be embedded
into the order-reduct of some specialization semilattice.
Then use the arguments for (C1) - (C5).

It is an open problem to provide an explicit description of  
the structure $ \widetilde{\mathbf  A}' $
given by Lemma \ref{exp} in cases (C1), (C3) and
(C6) - (C7). [In the meantime case (C2) has
been solved, manuscript to be posted on ArXiv]

\arxiv{
\begin{remark} \labbel{rmkexp}
If $\mathscr L' \setminus \mathscr L$ 
in Lemma \ref{exp} has only relation symbols, then the structure  
$\widetilde{\mathbf  A}'$ can be easily described.
 It has the same domain of $\mathbf A$ and, for every 
relation symbol $R \in \mathscr L' \setminus \mathscr L$
and every $a_{1}, \dots, a_{n} \in A$,
$R(a_{1}, \dots, a_{n})$ holds in  $\widetilde{\mathbf  A}'$
if and only if, for every extension of $\mathbf A$  to some  
$\mathbf  B$ which is a reduct of some $\mathbf  B' \in \mathcal K'$,
$R(a_{1}, \dots, a_{n})$ holds in  $\mathbf  B'$.
 
Henceforth, Lemma \ref{exp} is mainly interesting when 
$\mathscr L' \setminus \mathscr L$ has function symbols.

In this case, it may happen that some model $\mathbf A$ 
has some partial $\mathscr L'$-structure.
For example, in the situation at hand in the present note,
some element $a $ in the specialization semilattice $\mathcal S$ 
might have a closure $K_{_\mathbf S}a$.
 This eventuality does not influence
Theorem \ref{univt}, as stated, and, if $a \neq 0$, then  in    
$\widetilde{\mathbf  S}'$ $\upsilon_{_\mathbf S}(a)$ 
always acquires a new closure, 
different from $\upsilon_{_\mathbf S}(K_{_\mathbf S}a)$.
 
However, we can apply Lemma \ref{exp} even in this situation, when
we want to preserve the existing partial structure
(or, at least, some partial structure) in $\mathbf A$.
For example, as far as specialization semilattices  are concerned,
we  can add a new binary relation $R_K$ both to $\mathscr L$ 
and to $\mathscr L'$, and ask that   
\begin{align*}
&\forall xyz (R_K(x,y)   \Rightarrow ( z \sqsubseteq x 
\Leftrightarrow z \leq y) ), \  \text{ respectively,} 
\\
&\forall xy (R_K(x,y)   \Rightarrow  Kx=y )
  \end{align*}    
hold in $\mathcal K$ and, respectively, $\mathcal K'$. 
Using Lemma \ref{exp} in this situation,
we get a universal structure 
with respect to homomorphisms which preserve 
those closures witnessed by $R_K$. 

In the general case, to every  
$n$-ary function symbol $f$ in $\mathscr L \setminus \mathscr L'$,
one must associate an $n{+}1$-ary relation symbol
$R_f$, asking that the models in $\mathcal K'$ satisfy
\begin{align*}
&\forall x_1 \dots x_n y (R_f(x_1,\dots, x_n, y)  
 \Rightarrow  f(x_1,\dots, x_n)=y), 
  \end{align*}    
 while the axioms for $R_f$ in $\mathcal K$ depend
on the kind of structure one wants to preserve. 

Of course, it must be checked that
every model in $\mathcal K$ is actually 
a subreduct of some model in $\mathcal K'$.  
 \end{remark}   
} 

\section*{Acknowledgement}
We thank an anonymous referee for many useful
comments which helped to improve the paper.

\arxiv{

\section*{Appendix} \labbel{app} 

In this appendix we present a generalization
of \cite[Lemma 4.7]{mtt} which we have used 
in a tentative version of the present note.
The lemma turned out to be unnecessary in the subsequent
versions, but might be useful in different situations.

\begin{lemma} \labbel{lemq}
Suppose that $\mathbf S$ is a specialization semilattice 
and $ \sim$ is an equivalence relation on $S$ 
such that 
  \begin{enumerate}  
  \item 
$ \sim$ is a congruence for the semilattice reduct of $\mathbf S$, and
\item
If $a,b,b_1,c \in S$
are such that   $ a \sqsubseteq b \sim b_1 \sqsubseteq c$,
then there are 
$a_1 \sim a$ and  $c_1 \sim c$  in $S$ such that  
$a_1 \sqsubseteq c_1$.
  \end{enumerate}

Then $S/{\sim}$
can be given the structure of a specialization semilattice 
by considering the standard semilattice quotient and setting
\begin{equation}\labbel{q}
\text{$[a] \sqsubseteq [b]$ if 
there are $a_1,b_1 \in S$ such that 
$ a \sim a_1$, $ b \sim b_1$ and $ a_1\sqsubseteq b_1$.}   
   \end{equation}   

 Moreover, the projection 
$\pi: S \to S/{\sim} $  is a homomorphism
of specialization semilattices.  
 \end{lemma} 

\begin{proof} 
By classical arguments 
$\mathbf S/{\sim}$
is a semilattice
and the projection is a 
homomorphism
of specialization semilattices.
Hence it remains to prove \eqref{s1} - \eqref{s3}.

If $ [a] \leq [b]$
in  $\mathbf S/{\sim}$,
then, by the above paragraph, 
$ [a \vee  b] = 
[a] \vee [b]= [b] $,
that is,
$ a \vee  b \sim b $.
Taking $a_1 = a$ and 
$b_1 = a \vee  b $, we get 
$ [a] \sqsubseteq  [b]$
by \eqref{q}, since  
$  a \sqsubseteq a \vee  b $ in $\mathbf S$.
We have proved \eqref{s1}.  

If $[a] \sqsubseteq [b]$ 
and 
$[b] \sqsubseteq [c]$,
then by \eqref{q} 
there are $a^*,b^*, b_1^*, c^* \in S$ such that 
$ a \sim a^*$, $ b \sim b^*$,
$ b \sim b_1^*$, $ c \sim c^*$
 and $ a^*\sqsubseteq b^*$,
$ b_1^* \sqsubseteq c^*$.
By transitivity and symmetry of $ \sim$,
we get $ b^* \sim b_1^* $,
hence by item (2)
there are  
$a_1 \sim a^*$ and  $c_1 \sim c^*$  in $S$ such that  
$a_1 \sqsubseteq c_1$.
Again by transitivity and symmetry of $ \sim$,
$ a \sim a_1$ and  $ c \sim c_1$
hence $[a] \sqsubseteq [c]$ 
follows from \eqref{q}.
We have proved \eqref{s2}.  

Now suppose that 
$[a] \sqsubseteq [b]$ and
$[a^*] \sqsubseteq [b]$,
thus there are 
$a_1,b_1,a^*_1,b_2 \in S$ such that 
$ a \sim a_1$, $ b \sim b_1$
$ a^* \sim a_1^*$, $ b \sim b_2$,
$ a_1\sqsubseteq b_1$
and  $a_1^* \sqsubseteq b_2$.
By \eqref{s7}  
$a_1 \vee a_1^* \sqsubseteq b_1 \vee b_2$,
hence 
$[a] \vee [a^*] \sim 
[a_1] \vee [a_1^*] = [a_1 \vee a_1^*]
\sqsubseteq [b_1 \vee b_2] 
= [b_1] \vee [b_2] = [b] \vee [b] = [b]$,
since $ \sim$ is a semilattice congruence.
This completes the proof of \eqref{s3}.  

By the definition of $ \sqsubseteq $ on
$\mathbf S/{\sim}$ the projection is
a specialization homomorphism.
\end{proof}    
} 

\end{document}